\documentclass[12pt,a4paper]{amsart}

\usepackage{amsfonts, amsmath, amssymb, amsthm, amscd, hyperref}

\usepackage{a4wide}

\usepackage[pdftex]{graphicx} 

\newtheorem{thm}{Theorem}
\newtheorem*{thm*}{Theorem}
\newtheorem{lem}{Lemma}

\theoremstyle{definition}
\newtheorem{defn}{Definition}

\theoremstyle{remark}

\DeclareMathOperator{\conv}{conv}

\DeclareMathOperator{\cl}{cl}

\DeclareMathOperator{\rot}{rot}
\DeclareMathOperator{\aco}{aco}

\newcommand{\head}{\mathop{\tt head}}
\newcommand{\tail}{\mathop{\tt tail}}
\newcommand{\conc}{\circ}

\renewcommand{\int}{\mathop{\rm int}}

\renewcommand{\epsilon}{\varepsilon}

\begin{document}

\title{A measure of non-convexity in the plane and the Minkowski sum}

\author{R.N.~Karasev}
\thanks{This research was partially supported by the Dynasty Foundation.}

\email{r\_n\_karasev@mail.ru}
\address{
Roman Karasev, Dept. of Mathematics, Moscow Institute of Physics
and Technology, Institutskiy per. 9, Dolgoprudny, Russia 141700}

\keywords{Minkowski sum, weak convexity}
\subjclass[2000]{52A10, 52A30}

\begin{abstract}
In this paper a measure of non-convexity for a simple polygonal region in the plane is introduced. It is proved that for ``not far from convex'' regions this measure does not decrease under the Minkowski sum operation, and guarantees that the Minkowski sum has no ``holes''.
\end{abstract}

\maketitle

\section{Introduction}

Let us state the definition of the Minkowski sum of two sets $A, B\subset\mathbb R^d$.

\begin{defn} The \emph{Minkowski sum} is
$$
A+B = \{a+b : a\in A,\ b\in B\}.
$$
\end{defn}

In this paper we consider Minkowski sums in the plane. It is well-known, that the Minkowski sum of two convex sets is again convex. In the case of convex polygons it is computed by a simple ``edge merging and slope sorting'' algorithm. If we consider non-convex polygons, the computation of the Minkowski sum may require more complicated algorithms, see~\cite{afh2000, bdd2001, bhkw2007} for example. In the cited papers the problem of finding the Minkowski sum arised from packing or motion planning problems. Indeed, the set of possible shifts of a region $A$, that intersect another region $B$ is the Minkowski sum $(-A)+B$, where minus denotes the reflection w.r.t. the origin.

The most straightforward way to find the Minkowski sum of non-convex regions is to partition every non-convex region into convex polygons, calculate Minkowski sums of parts, and then take the union~\cite{afh2000}. In some practical applications, where the regions are essentially non-convex, this approach can be too complicated, in such cases it is convenient to use the intuitive ``orbital'' (or ``sliding'') methods, see~\cite{bdd2001,bhkw2007}. The latter methods deal with non-convex regions quite well, but the essential (and hard) part of these methods is finding ``holes'' in the Minkowski sum. Therefore, it is important to give a computable criterion for the Minkowski sum to have no ``holes''.

In section~\ref{definitions} we define a measure of non-convexity $\aco K$ for a simply-connected polygonal region (simple polygon) $K$ in the plane so that $\aco K \le 0$ in general and $\aco K = 0$ iff $K$ is convex. This measure of non-convexity uses some essential properties of $\mathbb R^2$, some other definitions of non-convexity measures valid for spaces of arbitrary (even infinite) dimensions are reviewed in~\cite{iva2005}. Another general non-convexity measure based on the path metric in $K$ can be found in~\cite{pan2001}.

The main result is stated as follows.

\begin{thm}
\label{aconv-add}
Suppose $K$ and $L$ are simple polygons, and $\aco K, \aco L > -\pi$. Then $K+L$ is a simple polygon and 
$$
\aco (K+L) \ge \min\{\aco K, \aco L\}.
$$
\end{thm}

This theorem shows that the property $\aco K > -\pi$ is stable under the Minkowski sum, and in this case the sum of an arbitrary number of simple polygons is simple. 

In the proof of Theorem~\ref{aconv-add} we use the following fact, which has its own value.
It generalizes the separation theorem for convex sets.

\begin{thm}
\label{aconv-sep}
Suppose $K$ is a simple polygon with $\aco K > -\pi$. Then for any point $x\not\in K$ there exist an angular region $A$ with apex $x$, such that $\angle A = \pi + \aco K$ and 
$$
A\cap K= \emptyset.
$$
\end{thm}

\section{Definition of angular convexity}
\label{definitions}

Now we make some definitions and fix some notation.

\begin{defn}
The sequence of points $v_1, \ldots, v_n\in \mathbb R^2$ (vertices) and the corresponding segments $v_1v_2, \ldots, v_{n-1}v_n$ (edges) is called a \emph{polyline}. We require that the consecutive vertices do not coincide $v_i\not=v_{i+1}$. 

For a polyline $P=v_1\ldots v_n$ we call the sequence of vectors $v_2-v_1,\ldots, v_n-v_{n-1}$ the \emph{shift sequence} and denote it $S(P)$.
\end{defn}

\begin{defn}
We call a polyline $v_1v_2\ldots v_{n+1}$ \emph{closed} if $v_1=v_{n+1}$. In this case we often index its vertices modulo $n$. 
\end{defn}

\begin{defn}
We call a polyline \emph{simple} if it has no self-intersections, i.e. its edges may intersect in one point if they are consecutive, otherwise they do not intersect. In a closed polyline we, of course, allow the first vertex to equal the last one.
\end{defn}

\begin{defn}
We call a compact set $K\subset\mathbb R^2$ a \emph{simple polygon} if its boundary is a closed simple polyline.
\end{defn}

It is obvious that generally the Minkowski sum of two simple polygons can be not simple, speaking informally it can have ``holes''. The objective of this paper is to find some sufficient conditions on the polygons that guarantee that the Minkowski sum is simple. We need some more definitions.

\begin{defn}
For two plane vectors define the \emph{skew product}
$$
[v, w] = v_xw_y - v_yw_x.
$$
Two nonzero vectors $v, w$ are called \emph{parallel} if $v = \alpha w$. In this case, if $\alpha<0$ we call the vectors \emph{opposite}.
\end{defn}

\begin{defn}
For two non-opposite vectors $v, w$ denote $\angle (v, w)$ the angle between vectors with sign, positive if $[v,w] >0$, negative if $[v,w] <0$.
\end{defn}

\begin{defn}
If the shift sequence of a polyline $P$ does not contain a consecutive pair of opposite vectors, we call $P$ \emph{non-reverse}.
\end{defn}

\begin{defn}
Let a polyline $P$ with shift sequence $S(P)=(s_1, \ldots, s_n)$ be non-reverse. The \emph{rotation} of $P$ (the same as rotation of $S(P)$) is 
$$
\rot P = \sum_{i=1}^{n-1} \angle(s_i, s_{i+1}).
$$
If the polyline has $2$ vertices we put $\rot P = 0$.

If the polyline $P$ is closed, then (the indices are modulo $n$)
$$
\rot P = \sum_{i=1}^n \angle(s_i, s_{i+1}).
$$
\end{defn}

\begin{defn}
Let $P$ be a non-reverse polyline. Denote the \emph{angular convexity} of $P$
$$
\aco P = \min_{L\subseteq P} \rot L,
$$
where the minimum is taken over subpolylines $L\subseteq P$ (obtained from $P$ by removing some vertices from its front and/or its back). 
\end{defn}

\begin{defn}
Let $K$ be a simple polygon and $P=\partial K$ be oriented so that $\rot P = 2\pi$. Denote the \emph{angular convexity} of $K$
$$
\aco K = \min_{L\subseteq P} \rot L,
$$
where the minimum is taken over simple polylines $L\subseteq P$, oriented along $P$. 
\end{defn}

The angular convexity of a polygon $K$ is illustrated on its slope diagram (dependance of the slope angle on the anticlockwise parameter) in Figure~\ref{slope-diag}

\begin{figure}
\includegraphics{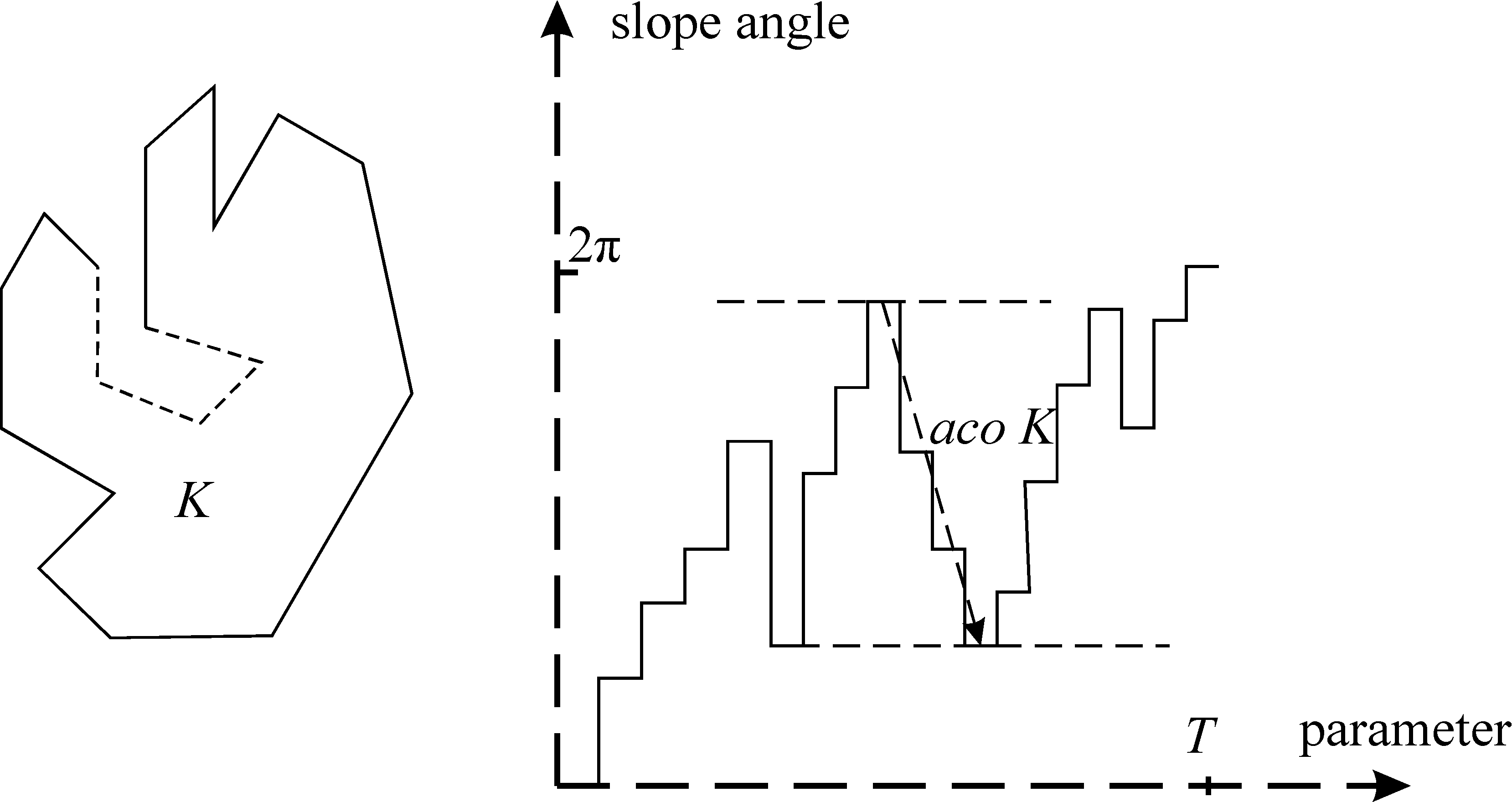}
\caption{Slope diagram and the angular convexity of a polygon, $T$ is the full length.}
\label{slope-diag}
\end{figure}

\section{The sorted sum of polylines}

Here we formulate and prove some lemmas that generalize the algorithm of finding Minkowski sum for convex polygons by edge sorting. We start from the definition. 

\begin{defn}
Denote the concatenation of sequences $S$ and $S'$ by $S\conc S'$. Denote the first element of nonempty $S$ as $\head S$, denote the sequence $S$ without its first element by $\tail S$.
\end{defn}

\begin{defn}
\label{sorted-sum}
Suppose two polylines $P=v_1,\ldots, v_{n+1}$ and $Q=w_1,\ldots, w_{m+1}$ are non-reverse, let their shift sequences be $S(P)=s_1, \ldots, s_n$ and $S(Q) =t_1, \ldots, t_m$. Suppose also that $\angle(s_1,t_1)= 0$ and $\rot P = \rot Q$.

Define the \emph{sorted sum} $R = P+_s Q$ as follows. Let the first vertex be $v_1 + w_1$, let the shift sequence $S(P+_s Q)$ be formed by the following rule: 

1) Start from $S_1 = S(P), S_2 = S(Q), S_3 = \emptyset$;

2) If $S_1=\emptyset$ and $S_2=\emptyset$ then put $S(P+_sQ)= S_3$ and quit;

3) If $S_1=\emptyset$ then put $S(P+_sQ) = S_3 \conc  S_2$ and quit;

4) If $S_2=\emptyset$ then put $S(P+_sQ) = S_3 \conc  S_1$ and quit;

5) If $\rot S_1 \ge \rot S_2$, then put $S_3 = S_3 \conc \head S_1$, $S_1 = \tail S_1$ and go to step $2$;

6) If $\rot S_1 < \rot S_2$, then put $S_3 = S_3 \conc  \head S_2$, $S_2 = \tail S_2$ and go to step $2$.
\end{defn}

Informally, we merge the shifts from two sequences in such a way, that from the two possibilities to choose the next shift we always choose the ``rightmost'' one, preferring the first sequence if the directions coincide. The condition that the starts and the ends of the shift sequences have the same direction, and the rotations are the same, can be relaxed in general, but it is crucial in Lemmas~\ref{sorted-sum-correct}, \ref{sorted-sum-rot}, and \ref{sorted-sum-aco}. An example of a sorted sum is shown in Figure~\ref{sorted-sum-pic}.

\begin{figure}
\includegraphics{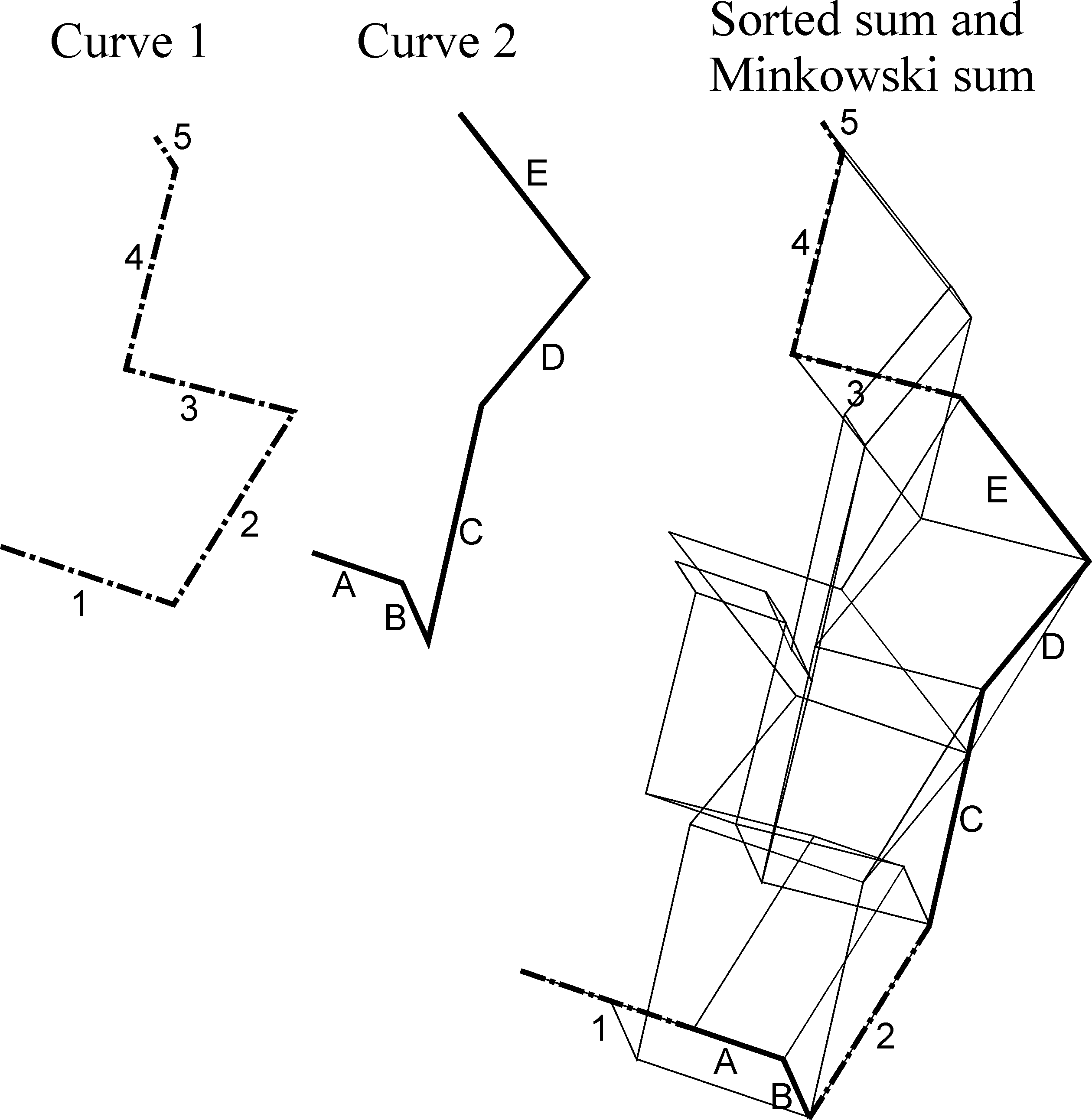}
\caption{The Minkowski sum and the sorted sum of two polylines.}
\label{sorted-sum-pic}
\end{figure}

First, we show how the sorted sum is related to the Minkowski sum.

\begin{lem}
\label{sorted-sum-param}
Let the polylines $P$ and $Q$ be as in Definition~\ref{sorted-sum} and let them have piece-wise linear parameterizations 
$$
P = \{p(t) : t \in [0, 1]\},\quad Q = \{q(t) : t \in [0, 1]\}.
$$
Then there are non-decreasing piece-wise linear functions (onto) $\phi, \psi : [0, 1] \to [0,1]$ such that
$$
r(t) = p(\phi(t)) + q(\psi(t))
$$
is a parameterization of $R$, in particular $P+_s Q \subseteq P+Q$ (the Minkowski sum of $P$ and $Q$ as subsets of the plane).
\end{lem}

\begin{proof}
Let $|S(P)| = n$, $|S(Q)| = m$, then $|S(P+_sQ)| = n+m$. Without loss of generality we may assume that the parameterization $p$ maps points $0, \frac{1}{n}, \frac{2}{n}, \ldots, \frac{n-1}{n}, 1$ to the vertices of $P$, and $q$ maps points $0, \frac{1}{m}, \frac{2}{m}, \ldots, \frac{m-1}{m}, 1$ to the vertices of $Q$.

The functions $\phi$ and $\psi$ will be continuous, piece-wise linear. We will construct them so that $\phi$ on a segment $\left[\frac{i-1}{n+m}, \frac{i}{n+m}\right]$ ($i=1, \ldots, n+m$) is either constant or linearly increases by $\frac{1}{n}$, and $\psi$ on the same segment is either constant or linearly increases by $\frac{1}{m}$. 

By definition put: if the $i$-th shift in $S(P+_sQ)$ is taken from $S(P)$, then $\phi$ increases and $\psi$ remains constant, if it is taken from $S(Q)$, $\phi$ remains constant and $\psi$ increases.

In the first case denote $S(P+_sQ) = S' \conc  s_i \conc  S''$, where $|S'| = i-1$, then for $t_1 = \frac{i-1}{n+m}$ and $t_2=\frac{i}{n+m}$ 
$$
p(\phi(t_1)) + q(\psi(t_1)) = v_1 +w_1 + \sum S'\cap S(P) + \sum S'\cap S(Q) = v_1 + w_1 + \sum S'\in P+_sQ,
$$
and 
$$
p(\phi(t_2)) + q(\psi(t_2)) = v_1 +w_1 + \sum S'\cap S(P) + s_k + \sum S'\cap S(Q) = v_1 + w_1 + \sum S' + s_i\in P+_sQ.
$$
The above equations mean that the points $p(\phi(t_1)) + q(\psi(t_1))$ and $p(\phi(t_2)) + q(\psi(t_2))$ are two consecutive vertices of the polyline $P+_sQ$, and for $t\in [t_1, t_2]$ the point $p(\phi(t)) + q(\psi(t))$ is on the corresponding segment of $P+_sQ$. 

The second case is considered similarly.
\end{proof}

\begin{lem}
\label{sorted-sum-correct} Suppose $P, Q, R$ are as in Definition~\ref{sorted-sum}. If $\aco P, \aco Q >-\pi$, then $R = P+_s Q$ is non-reverse.
\end{lem}

\begin{proof} 
Assume the contrary: its consecutive shifts $s_k$ and $t_l$ (indexed as they were in $P$ and $Q$) are opposite. Take the sequences $S_1$ and $S_2$ from the definition in the state before adding $s_k$ and $t_l$ to $S_3$, so $s_k = \head S_1$, $t_l = \head S_2$.

Since $s_k$ and $t_l$ are opposite $\rot S_1 = \rot S_2 + (2N+1) \pi$, for some $N\in\mathbb Z$. Since $s_k$ comes before $t_l$ in $S(P+_sQ)$, then $\rot S_1 \ge \rot S_2$. Thus
$$
\rot S_1 \ge \rot S_2 + \pi.
$$
If $|S_1| \ge 2$, then by the construction $\rot \tail S_1 \le \rot S_2$, so $\angle (\head S_1, \head\tail S_1) \ge \pi$, which cannot be the case. If $|S_1| = 1$, then $\rot S_2 \le -\pi$, which contradicts with the angular convexity condition.
\end{proof}

\begin{lem}
\label{sorted-sum-rot} Suppose $P, Q, R$ are as in Definition~\ref{sorted-sum}. If $\aco P, \aco Q >-\pi$, then $\rot R = \rot P = \rot Q$.
\end{lem}

\begin{proof}
Denote $\rot P = \rot Q = \alpha$. To prove $\rot S(P+_sQ) = \alpha$, it is sufficient to prove the following statement: if on some cycle of the construction the last element of $S_3$ is from $S(P)$, then $\rot (S_3 \conc  S_1) = \alpha$, if the last element of $S_3$ is from $S(Q)$, then $\rot (S_3 \conc  S_2) = \alpha$. 

Let us prove it by induction. If $|S_3|\le 2$ the statement is true. Let $S_3 = S'_3 \conc  x \conc y$, if $x$ and $y$ are shifts from $S(P)$, then 
$$
\rot (S_3 \conc  S_1) = \rot ( S'_3 \conc  x\conc y \conc  S_1) = \alpha
$$
by the inductive assumption. The same is true if $x,y\in S(Q)$. It is left to consider the case $x\in S(P), y\in S(Q)$. Consider two cases depending on whether $S_1$ is empty or not.

Case 1: $S_1$ is non-empty. By the construction of the sorted sum we have:
$$
\rot (x\conc  S_1) \ge \rot (y\conc S_2) \ge \rot S_1,
$$
and since $\rot (x\conc  S_1) = \angle(x, \head S_1) + \rot S_1$, then we conclude that 
$$
|\rot (x\conc  S_1) - \rot (y\conc  S_2)| < \pi.
$$

By the inductive assumption $\rot (S'_3 \conc  x\conc  S_1) = \alpha$. The difference 
\begin{multline*}
\rot (S_3 \conc  S_2) - \rot (S'_3 \conc  x\conc S_1) = \rot (x\conc y\conc S_2) - \rot (x\conc S_1) =\\
= \angle(x, y) + \rot (y\conc  S_2) - \rot(x\conc  S_1)
\end{multline*}
should be a multiple of $2\pi$. But its absolute value is less than $2\pi$, since $|\angle(x,y)|<\pi$ and $|\rot (x\conc  S_1) - \rot (y\conc S_2)| < \pi$. Hence $\rot (S_3 \conc  S_2)=\alpha$.

Case 2: $S_1$ is empty. Then, similar to the previous case:
$$
0 \ge \rot (y\conc  S_2),
$$
and  the difference
$$
\rot (S'_3 \conc  x\conc  y\conc  S_2) - \rot (S'_3 \conc x) = \rot (x\conc y\conc  S_2) = \angle(x, y) + \rot (y\conc  S_2)
$$
should be a multiple of $2\pi$. Since $|\angle(x,y)|<\pi$ and $0 \ge \rot (y\conc  S_2) \ge \aco Q$, this difference should be zero.
\end{proof}

\begin{lem}
\label{sorted-sum-aco} Suppose $P, Q, R$ are as in Definition~\ref{sorted-sum}. If $\aco P, \aco Q >-\pi$, then $\aco R \ge \min\{\aco P, \aco Q\}$.
\end{lem}

\begin{proof}
Denote $\alpha = \rot P = \rot Q = \rot R$, $\gamma = \min\{\aco P, \aco Q\}$. 

Assume the contrary: for some subsequence $S$ of the sequence $S(P+_sQ)$ we have $\rot S < \gamma$. It is equivalent to the following statement: the sequence $S_3$ in the construction of sorted sum takes two values $S'_3$ and $S''_3$ (in this order) so that 
$$
\rot S''_3 < \rot S'_3 +\gamma.
$$

Put $S'_3 = T' \conc  x$, $S''_3 = T'' \conc y$, and denote the corresponding values of $S_1$ and $S_2$ by $S'_1, S'_2, S''_1, S''_2$.

If $x$ and $y$ are from $S(P)$ then (see the proof of Lemma~\ref{sorted-sum-rot})
$$
\rot (x\conc  S'_1) = \alpha - \rot S'_3,\quad \rot (y\conc  S''_1) = \alpha - \rot S''_3
$$
and by the assumption
$$
\rot (x\conc  S'_1) < \rot (y\conc  S''_1) +\gamma.
$$
If $S'_1 = \Sigma \conc y\conc  S''_1$ then 
$$
\rot (x\conc  \Sigma \conc y) + \rot (y\conc  S''_1) < \rot (y\conc  S''_1) +\gamma,
$$
and therefore, $\rot (x\conc  \Sigma \conc y) < \gamma$, which is a contradiction with $\aco P\ge\gamma$. If $x$ and $y$ are both from $S(Q)$, the same contradiction is obtained.

Now assume that $x$ is from $S(Q)$, $y$ is from $S(P)$, then 
$$
\rot (x\conc  S'_2) = \alpha - \rot S'_3,\quad \rot (y\conc  S''_1) = \alpha - \rot S''_3
$$
and by the assumption
$$
\rot (x\conc  S'_2) < \rot (y\conc  S''_1) +\gamma.
$$
By the construction $\rot S'_1 \le \rot (x\conc  S'_2)$ (when $x$ was added to $S_3$), and
$$
\rot S'_1 < \rot (y\conc  S''_1) + \gamma.
$$
If $S'_1 = \Sigma \conc y\conc  S''_1$ then 
$$
\rot (\Sigma \conc y) + \rot (y\conc  S''_1) < \rot (y\conc  S''_1) +\gamma,
$$
and therefore, $\rot (\Sigma\conc y) < \gamma$, which is a contradiction with $\aco P\ge\gamma$.

The case $x$ is from $S(P)$, $y$ is from $S(Q)$ is considered the same way.
\end{proof}

\section{Elimination of self-intersections}

In this section we consider a polyline $P$ that has self-intersection. Then we convert it to a polyline without self-intersections and try to describe how the rotation of $P$ changes. First we need a definition.

\begin{defn}
A polyline $P$ is said to be in \emph{general position}, if 

1) all its vertices are distinct;

2) any two edges intersect in at most one point, the common point being either the common vertex, or in the relative interiors of both edges;

3) any three edges do not have a common point. 
\end{defn}

It is clear that, by arbitrarily small movements of the vertices, a polyline can be put to a general position.

\begin{defn}
Let a polyline $P$ be in general position. Suppose it is given by a piece-wise linear parameterization $P = \{p(t):t\in [a,b]\}$ and has a self-intersection $p(t_1) = p(t_2)$ for some $t_1 < t_2$. Call the \emph{loop removal} the transform that replaces $P$ by the concatenation of polylines $P' = \{p(t) : t\in[a, t_1]\}$ and $P'' = \{p(t) : t\in[t_2, b]\}$. 
\end{defn}

It is clear that generally $\rot P$ changes by a multiple of $2\pi$ after the loop removal. The following lemma tells more.

\begin{lem}
\label{loop-removal} Let $P$ be a polyline in general position. Suppose that the subpolyline of $P$ between $p(t_1)$ and $p(t_2)$ has rotation $>-\pi$. Then the rotation $\rot P$ cannot increase after loop removal.
\end{lem}

\begin{proof}
Let $p(t_1)$ lie on the edge with direction $s_1$, $p(t_2)$ line on another edge with direction $s_2$. Denote the subpolyline between $p(t_1)$ and $p(t_2)$ by $L$.

If we identify $p(t_1)$ and $p(t_2)$ in $L$ we obtain a closed polyline, so 
$$
\rot L + \angle(s_2, s_1) = 2\pi N,\quad N\in\mathbb Z.
$$
Since $\rot L > -\pi$ and $|\angle(s_2, s_1)|<\pi$, we obtain $\rot L + \angle(s_2, s_1) \ge 0$.

After removal of $L$ the rotation of the new polyline differs from $\rot P$ by $\angle(s_1, s_2) - \rot L = -\angle(s_2, s_1)-\rot L\le 0$, so it does not increase.
\end{proof}

\begin{lem}
\label{loop-elimination}
Let $P$ be a polyline in general position with $\aco P > -\pi$. Then we can remove all loops in $P$ so that its rotation does not increase.
\end{lem}

\begin{proof}
Consider all the self-intersections in $P$, choose the first (w.r.t. the parameterization) such point $p(t_1)=p(t_2)$, and remove the loop between $p(t_1)$ and $p(t_2)$. Then continue such steps until all the self-intersections are removed. It is clear that the removed segments of parameterization do not intersect, each removed subpolyline had angular convexity $>-\pi$ and by Lemma~\ref{loop-removal} the angular convexity never increased.
\end{proof}

An example of removing all loops is shown in Figure~\ref{loop-removal-pic}.

\begin{figure}
\includegraphics{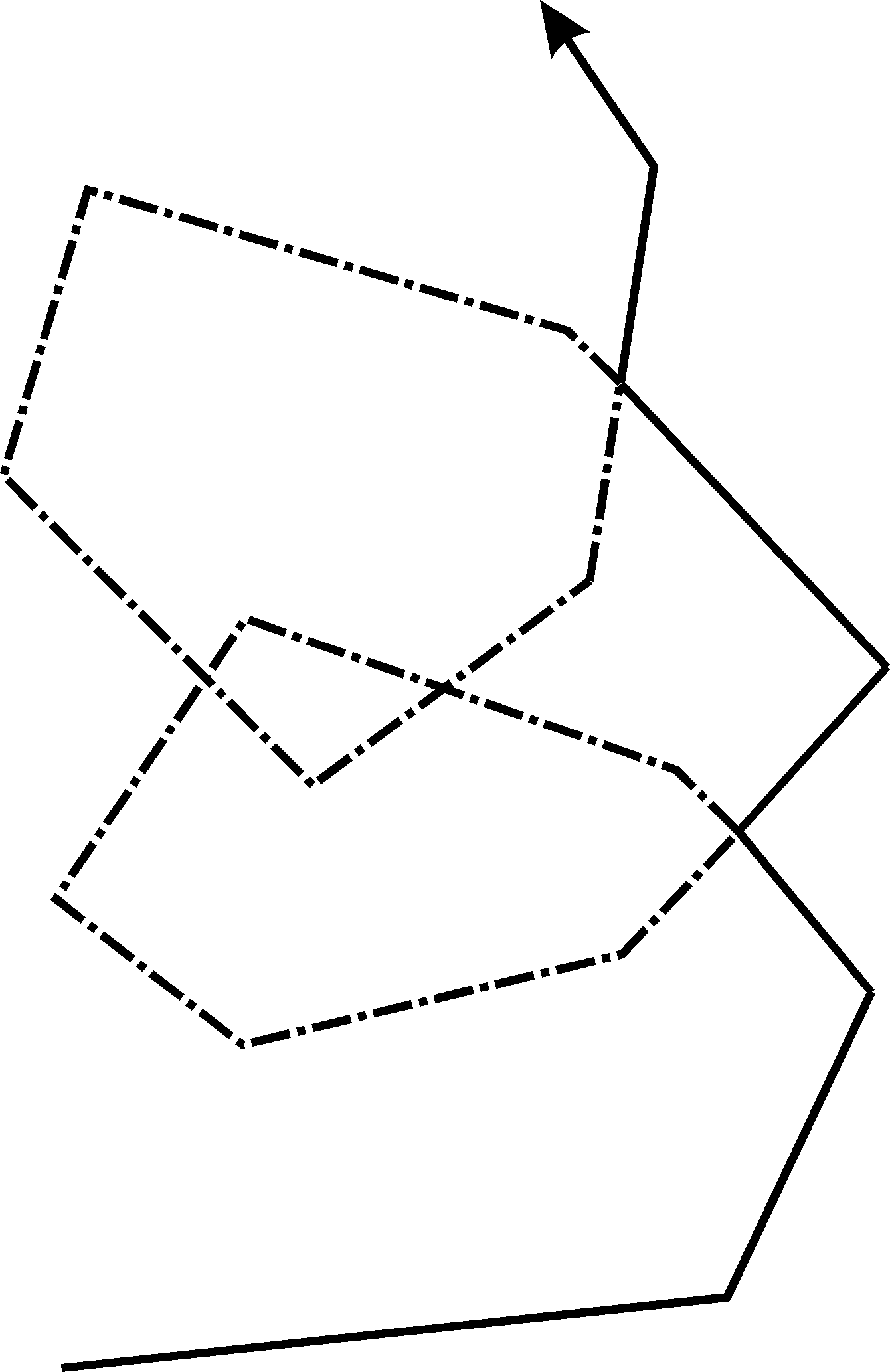}
\caption{Removing loops in a polyline.}
\label{loop-removal-pic}
\end{figure}

\section{Proof of Theorem~\ref{aconv-sep}}

Let us give some definitions.

\begin{defn}
A convex hull of two rays $r_1$ and $r_2$ with common apex is called an \emph{angular region}.
For an angular region $A$ we denote its \emph{angular measure} $\angle A = |\angle (r_1, r_2)|$.
\end{defn}

We generalize the notion of rotation to piece-wise smooth curves.

\begin{defn}
Let $C=\{c(t) : t\in [a, b]\}$ be a smooth curve with non-zero derivative. Then the unit tangent
$$
\tau(t) = \frac{c'(t)}{|c'(t)|}
$$
is well defined and can be considered as a continuous map $\tau : [a, b]\to S^1$ to the unit circle. Consider the universal cover $\kappa : \mathbb R\to S^1$ given by the anti-clockwise angle. The map $\tau$ is lifted continuously to a map $\tilde\tau : [a, b] \to \mathbb R$ so that $\tau = \kappa\circ\tilde\tau$. 

Then the \emph{rotation number} of $C$ is $\tilde\tau(b) - \tilde\tau(a)$. 
\end{defn}

\begin{defn}
Let $C=\{c(t) : t\in [a, b]\}$ be a piece-wise smooth curve, composed as a concatenation of smooth curves $C_1\conc\dots \conc C_n$. Denote the values of the parameter, where $C_i$ is changed to $C_{i+1}$ by $t_i$. We require the vectors $c'_-(t_i)$ and $c'_+(t_i)$ (left and right derivative) to be non-opposite.

Then the \emph{rotation number} of $C$ is
$$
\rot C = \sum_{i=1}^n \rot C_i + \sum_{i=1}^{n-1} \angle(c'_-(t_i), c'_+(t_i)).
$$ 
\end{defn}

It is clear that if the curve $C$ is approximated by a polyline $P$ with small enough step, then $\rot C = \rot P$. The angular convexity of a curve is defined similarly.

Now we are going to prove Theorem~\ref{aconv-sep}. The polyline $P=\partial K$ can be approximated by a smooth curve as follows: let $v_i$ be its vertex with nonzero angle $\angle(v_i-v_{i-1}, v_{i+1} - v_i)$, in a small neighborhood of $v_i$ we replace the union of two small equal segments of $P$ (the first segment is from $v_i$ back along $\partial P$, the second is from $v_i$ forth along $\partial P$) by a circular arc $A_i$, so that the tangent becomes continuous along the resulting curves. If the segments are taken small enough, then the new boundary $P'$ has no self-intersections. If we denote by $K'$ the region, bounded by $P'$, then the (tested for the separation property) point $x$ still lies outside $P'$ for close enough approximation.

Note that $\rot A_i = \angle(v_i-v_{i-1}, v_{i+1} - v_i)$ by the construction, so it is clear that $\aco P' = \aco P = \gamma$.

For any two points $p_1, p_2\in P'$ let us denote $[p_1, p_2]_{P'}$ the subcurve of $P'$, oriented along $P'$, that starts at $p_1$ and ends at $p_2$. Consider the following functions of a point $p\in P'$
$$
\gamma_+(p) = \min_{q\in P'} \rot[p, q]_{P'},\quad \gamma_-(p) = \min_{q\in P'} \rot[q, p]_{P'}.
$$
It is clear that they are continuous and $\gamma_-(p) + \gamma_+(p) \ge \gamma$.

Take the rays $r_+(p)$ and $r_-(p)$ with apex $p$ so that 
$$
\angle(\tau(p), r_+(p)) = \gamma_+(p),\quad \angle(r_-(p), \tau(p)) = \gamma_-(p).
$$
Note that the rays $r_+(p)$ and $r_-(p)$ point outside or tangentially to $K'$, and the angle between them is at least $\pi + \gamma > 0$. Denote the angular region $A(p)=\conv (r_+(p)\cup r_-(p))$. 

We are going to show that $K'\cap \int A(p)=\emptyset$. Assume the contrary, in this case $P'$ has to pass through the interior of $A(p)$. The situation is outlined in Figure~\ref{separation-pic}.

\begin{figure}
\includegraphics{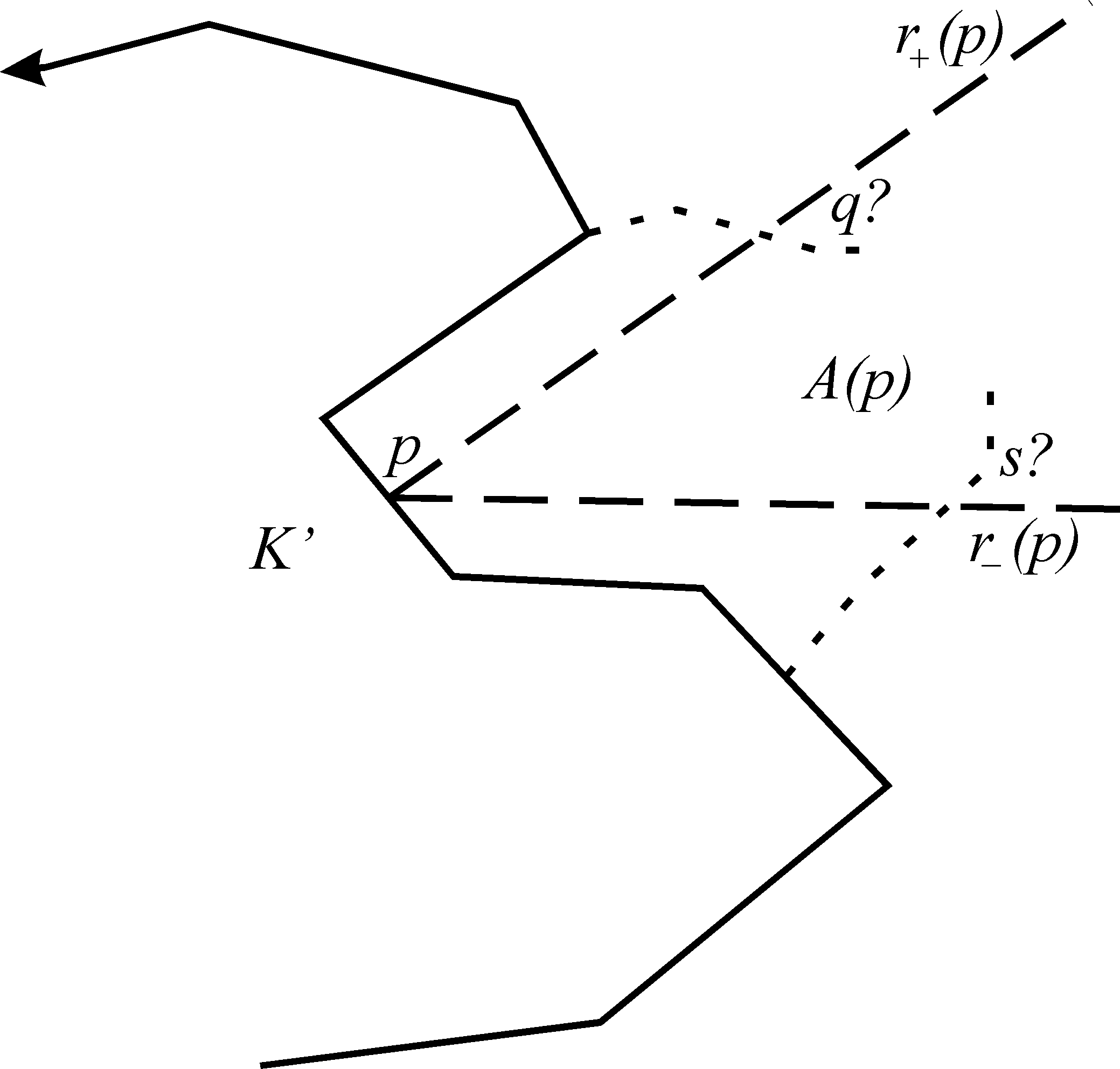}
\caption{The polygon $K'$ and the angular region $A(p)$.}
\label{separation-pic}
\end{figure}

Let us move the point $q$ from $p$ along $P'$. Note the first time $q$ gets inside $\int A(p)$. Suppose $q$ intersects the ray $r_+(p)$, then the rotation of the closed (piece-wise smooth) non-self-intersecting curve $[p,q]_{P'} \conc [q,p]$ (the last segment is a straight line segment) is $-2\pi$, and it is clear that $\rot [p, q]_{P'}$ should be less than $\gamma_+(p)$, which is a contradiction.

The only possibility left is that $q$ (when moving from $p$ along $P'$) gets into $A(p)$ across $r_-(p)$. Let $q$ be the first such point on $r_-(p)$, in this case we consider the closed non-self-intersecting curve $[p, q]_{P'} \conc  [q, p]$ that bounds a region $L$. If we move a point $s$ backwards from $p$ along $P'$, then it should go out of $L$, in fact it has to intersect $[q,p]$ and get into $\int A(p)$. Let $s$ be the first such point on $[q,p]$, in this case the closed simple curve $[s,p]_{P'} \conc  [p, s]$ has rotation $-2\pi$, and $\rot [s, p]_{P'} < \gamma_-(p)$, which is a contradiction again. 

Thus the theorem is proved for points $x$, close enough to $P'$, now we prove it for any $x$ outside $P'$. For $p\in P'$ denote $\nu(p)$ the direction of the bisector of $A(p)$, and $\mu(p)$ the unit direction of $x-p$. They give two maps $\nu : P'\to S^1$ and $\mu : P'\to S^1$. The map $\nu$ is homotopic to the Gauss map of $P'$, so $\deg \nu = 1$. The map $\mu$ is homotopic to a constant map (the homotopy is obtained by moving $x$ to infinity without crossing $P'$), so $\deg \mu = 0$. Thus $\nu$ and $\mu$ should coincide on some $p$. The other (more elementary) way to prove the coincidence is to note that the rotation of the vector $\nu(p)$ when $p$ moves along $p'$ is $2\pi$, and the rotation of $\mu(p)$ is zero.

It is easy to see that for such $p$ the point $x$ lies inside $A(p)$, and the angular region $A(p) + x - p\subset\int A(p)$ gives what we need. 

\section{Proof of Theorem~\ref{aconv-add}}

Denote $\gamma = \min\{\aco K, \aco L\}$. Consider the set $M=K+L$. Later we prove that it is simply connected (has no \emph{holes}, i.e. bounded connected components of the complement $\mathbb R^2\setminus M$), but now we assume that the holes may exist and denote its outer piece of boundary by $B$ (the boundary of the component of $\infty$ of $\mathbb R^2\setminus M$). We are going to show that $\aco B\ge \gamma$.

Suppose that for some two points $p, q\in B$ we have $\rot [p, q]_B = \delta < \gamma$ (the notation is the same as in the proof of Theorem~\ref{aconv-sep}). We may assume that $p$ and $q$ lie inside the edges of $B$. In this case from the definition of the Minkowski sum we deduce that
$$
p = a+b,\ a\in K,\ b\in L,\quad q = c+d,\ c\in K,\ d\in L.
$$

One of the points $a$ and $b$ (let it be $a$) lies on an edge (since $p$ lies on an edge) of $\partial K$  with outer normal $\nu$. The other point $b$ either lies on an edge of $L$ with the same normal, or is a vertex of $L$. 

In the latter case the inner product $(\nu, x)$ on some small enough neighborhood of $b$ in $L$ attains its maximum exactly at $b$, in this case we can insert a ``virtual edge'' to $L$ of length zero at point $b$, so that it has normal $\nu$. This virtual edge does not affect the rotation of any subpolyline of $\partial L$. In the sequel we assume that $a, b, c, d$ lie on (possibly virtual) edges, the edges of $a, b, p$ have the same direction, and the edges of $c,d,q$ have the same direction.   

Note that if $\aco \partial K>-\pi$, then for any subpolyline $[x, y]_{\partial K}$ its rotation is less than $3\pi$, otherwise the subpolyline $[y,x]_{\partial K}$ would have the rotation $\le-\pi$. Now we see that $\rot [a, c]_{\partial K} = \delta + 2\pi N$, $N\in\mathbb Z$. From the angular convexity condition and the above note we deduce that $\rot [a, c]_{\partial K} = \delta + 2\pi$. So we have 
$$
\rot[c, a]_{\partial K} = \rot[d, b]_{\partial L} = -\delta.
$$

Now consider the sorted sum $S = [c, a]_{\partial K}+_s [d, b]_{\partial L}$. By Lemma~\ref{sorted-sum-rot} and Lemma~\ref{sorted-sum-aco} we have $\rot S = -\delta$, $\aco S \ge \gamma > -\pi$. Lemma~\ref{sorted-sum-param} shows that $S\subseteq M$. The situation is illustrated in Figure~\ref{sum-aco-pic}.

\begin{figure}
\includegraphics{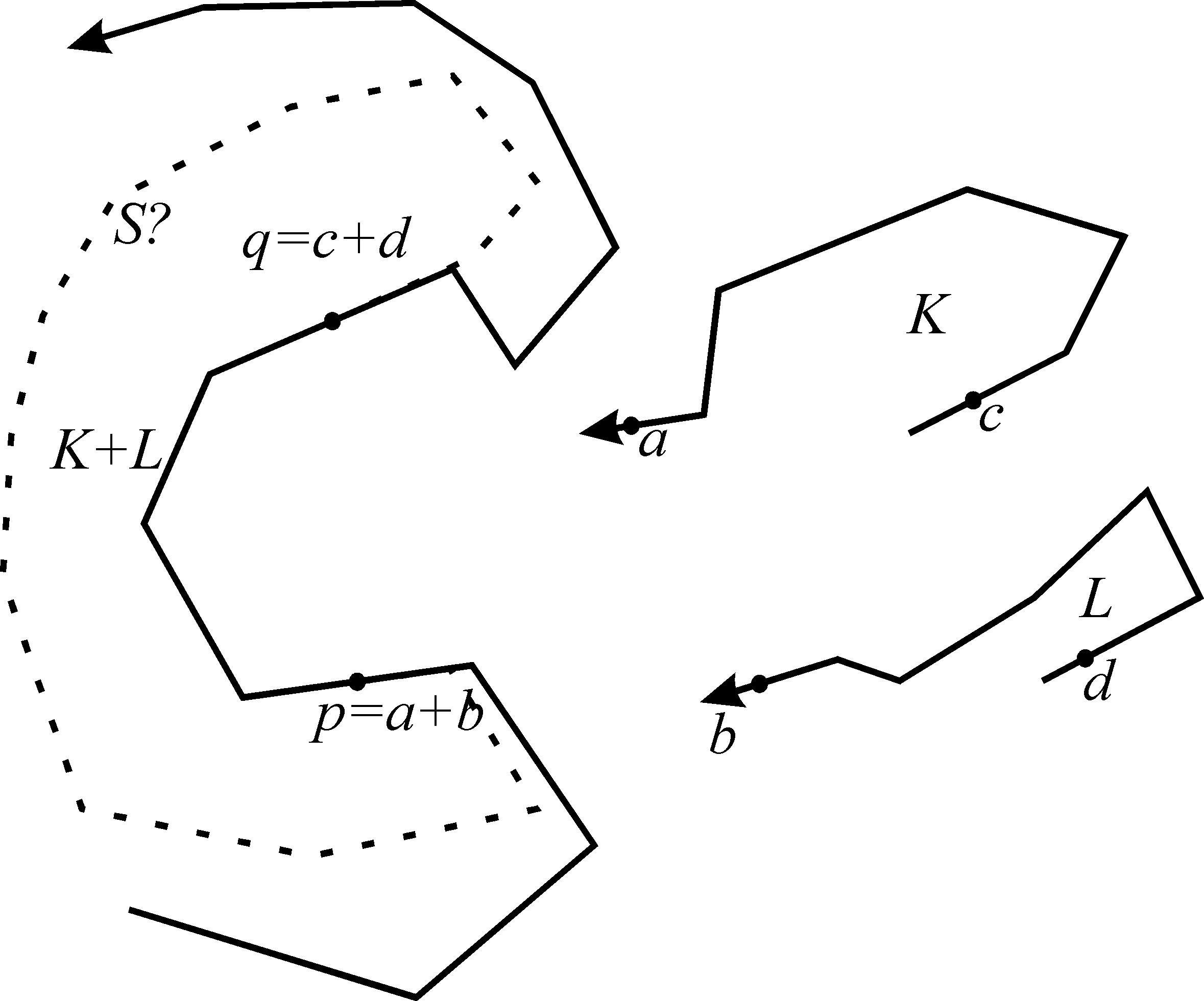}
\caption{Points $p=a+b$, $q=c+d$, and the polyline $S=[c, a]_{\partial K}+_s [d, b]_{\partial L}$.}
\label{sum-aco-pic}
\end{figure}

Note that $M=\cl\int M$, $\int M$ is connected, so we may put $S$ to general position so that it remains in $M$ and the only points of $S\cap \partial M$ are $p$ and $q$, we assume that the rotation of $S$ under this perturbation changes by arbitrarily small value, and $\aco S$ remains $>-\pi$. 

Then we eliminate loops on $S$ by Lemma~\ref{loop-elimination} and obtain a simple polyline $S_0\subseteq M$ with $\rot S_0 \le -\delta$. From Lemma~\ref{sorted-sum-param} we know that $S_0$ goes from $q$ to $p$. Moreover, by the construction of the sorted sum its edge directions at $p$ and $q$ coincide (up to some small change on passing to general position) with the edge directions of $[p, q]_B$ (edges may be virtual). So the closed polyline $C = S_0\conc [p,q]_B$ has rotation number $\le 0$ and does not have self-intersection. 

Hence $\rot C = -2\pi$, and the region that it bounds is to the right of $C$. But $C\subset M$, and along the polyline $[p, q]_B$, the set $M$ is to the left of $[p, q]_B$ ($B$ is oriented anti-clockwise), which is a contradiction. Thus we have proved that $\aco B \ge \gamma$.

Now we are going to prove that $M$ actually has no holes. Assume the contrary: a point $x\in\mathbb R^2\setminus M$ is not connected to $\infty$ in the complement of $M$. Let us perturb the vertices of $K$ and $L$ by values $<\delta$. If $\delta$ is chosen small enough then $\aco K, \aco L>-\pi$ condition is kept, and $x$ remains outside $M$. Moreover, if $M$ does not have a hole containing $x$ after perturbation, then the complement $\mathbb R^2\setminus M$ contains by Theorem~\ref{aconv-sep} an angular region $A$ with apex $x$ and bounded from below measure $\ge(\pi+\gamma)/2$. Hence for small enough $\delta$ the region $A$ cannot be ``blocked'' by returning to the non-perturbed $K$ and $L$, we have a contradiction. Thus we have proved that for small enough ($<\delta$) perturbations of $K$ and $L$ their sum $M$ still has a hole. i.e. the property of having a hole is open.

We are going to consider the parameterized family of polygons 
$$
M(t) = K+tL.
$$
Note that the edges of $M(t)$ are contained in the lines (we define lines by two points)
$$
\{(v_1 + tw_1, v_2+tw_1), (v_1 + tw_1, v_1+tw_2)\},
$$
where $v_1,v_2\in K$ and $w_1,w_2\in L$ are vertices. By perturbing $K$ and $L$ we may require that all such lines (there is a finite number of them) have different directions. We also require that when $t$ changes in the range $(0,+\infty)$, no $4$ of the lines meet at a single point, this can be achieved by perturbation.

If the parameter $t$ is close enough to zero, then the polygon $M(t)$ has the following structure: near an edge of $K$ it has a ``long'' edge, and near a vertex of $K$ it may be complicated. But since $K$ is an angular region $K\cap U(v)$ (possibly concave) in some neighborhood of a vertex $v\in K$, then $M(t)$ in $U(v)$ (and for small enough $t$) is a union of a family of translates of $K\cap U(v)$, that is starshaped. Hence, for small enough $t$, locally $M(t)$ has the same topology as $K$, and cannot have holes in global. 

Now consider 
$$
t_0 = \inf \{t : K+tL\ \text{has a hole}\},
$$
it is already proved that $t_0>0$, assume $t_0<+\infty$. Since the ``has hole'' property is open, $M(t_0)$ has no holes. Consider what could happen if we increase $t$ by a small value $<\delta$. From the general position assumption $M(t)$ may be changed by adding some extra edge instead of a vertex, but this cannot give a hole. Hence $t_0=+\infty$ and the proof is complete.

\end{document}